%% file: relbo.tex
\documentclass[11pt, reqno]{amsart}
\usepackage[utf8x]{inputenc}
\usepackage{amsmath, amsthm, amscd, amssymb, cite, comment, datetime, enumerate, textcomp, units, stmaryrd, wasysym, extarrows}
\usepackage[all, arrow]{xy}\SelectTips{cm}{10}
\makeatletter
\def\@setcopyright{}
\def\serieslogo@{}
\makeatother
\numberwithin{equation}{section}
\begin{document} 
\input{rbheaders.tex}
\title[Relative Bogomolov extensions]{Relative Bogomolov extensions}
\author[Robert Grizzard]{Robert Grizzard}
\address{\newline University of Texas at Austin \newline Department of Mathematics, RLM 8.100 \newline 2515 Speedway, Stop C1200 \newline Austin, TX 78712-1202}
\email{rgrizzard@math.utexas.edu}
\today~~\currenttime~~CDT
\begin{abstract}
A subfield $K \subseteq \QQbar$ has the Bogomolov property if there exists a positive $\varepsilon$ such that no non-torsion point of $K^\times$ has absolute logarithmic height below $\varepsilon$.  We define a relative extension $L/K$ to be \emph{Bogomolov} if this holds for points of $L^\times \setminus K^\times$.  We construct various examples of extensions which are and are not Bogomolov.  We prove a ramification criterion for this property, and use it to show that such extensions can always be constructed if some rational prime has bounded ramification index in $K$ and $K/\QQ$ is Galois. 
\end{abstract}
\maketitle
\section{Introduction}\label{intro}
\input{rbintro.tex}



\section{Lower bounds and a ramification criterion for (RB)}\label{lower}
\input{rblower.tex}

\section{Adjoining $\ell \nth$ roots and the proof of Theorem \ref{finram}}\label{lth}
\input{rblth.tex}

\section{Examples and the proof of Theorem \ref{4cases}}\label{examples}
\input{rbexamples.tex}

\bibliography{b}{}
\end{document}

%% file: rbheaders.tex
\baselineskip=14.5pt

\bibliographystyle{plain}

\makeatletter
\def\imod#1{\allowbreak\mkern10mu({\operator@font mod}\,\,#1)}
\makeatother
\newcommand{\vectornorm}[1]{\left|\left|#1\right|\right|}
\newcommand{\ip}[2]{\left\langle#1,#2\right\rangle}
\newcommand{\ideal}[1]{\left\langle#1\right\rangle}
\newcommand{\sep}[0]{^{\textup{sep}}}
\newcommand{\Span}[0]{\operatorname{Span}}
\newcommand{\Tor}[0]{\operatorname{Tor}}
\newcommand{\Stab}[0]{\operatorname{Stab}}
\newcommand{\Orb}[0]{\operatorname{Orb}}
\newcommand{\Soc}[0]{\operatorname{Soc}}
\newcommand{\Kernel}[0]{\operatorname{ker }}
\newcommand{\Gal}[0]{\operatorname{Gal}}
\newcommand{\Aut}[0]{\operatorname{Aut}}
\newcommand{\Out}[0]{\operatorname{Out}}
\newcommand{\Inn}[0]{\operatorname{Inn}}
\newcommand{\lcm}[0]{\operatorname{lcm}}
\newcommand{\Image}[0]{\operatorname{im }}
\newcommand{\vol}[0]{\operatorname{vol }}
\newcommand{\pgl}[0]{\operatorname{PGL}_2(\mathbb{F}_3)}
\newcommand{\AGL}[0]{\operatorname{AGL}}
\newcommand{\PGL}[0]{\operatorname{PGL}}
\newcommand{\PSL}[0]{\operatorname{PSL}}
\newcommand{\PSp}[0]{\operatorname{PSp}}
\newcommand{\PSU}[0]{\operatorname{PSU}}
\newcommand{\Sp}[0]{\operatorname{Sp}}
\newcommand{\POmega}[0]{\operatorname{P\Omega}}
\newcommand{\Uup}[0]{\operatorname{U}}
\newcommand{\Gup}[0]{\operatorname{G}}
\newcommand{\mcFup}[0]{\operatorname{F}}
\newcommand{\Eup}[0]{\operatorname{E}}
\newcommand{\Bup}[0]{\operatorname{B}}
\newcommand{\Dup}[0]{\operatorname{D}}
\newcommand{\Mup}[0]{\operatorname{M}}
\newcommand{\mini}[0]{\operatorname{min}}

\newcommand{\maxi}[0]{\textup{max}}
\newcommand{\modu}[0]{\textup{mod}}
\newcommand{\nth}[0]{^\textup{th}}
\newcommand{\sd}[0]{\leq_{sd}}
 \newtheorem{theorem}{Theorem}[section]
 \newtheorem{proposition}[theorem]{Proposition}
 \newtheorem{lemma}[theorem]{Lemma}
 \newtheorem{corollary}[theorem]{Corollary}
 \newtheorem{conjecture}[theorem]{Conjecture}
 \newtheorem{definition}[theorem]{Definition}
 \newtheorem{question}[theorem]{Question}
 \newtheorem*{claim*}{Claim}
 \newtheorem{claim}[theorem]{Claim}
 \newtheorem{example}[theorem]{Example}
 \newtheorem*{example*}{Example}
 \newtheorem{remark}{Remark}
 \newtheorem*{remark*}{Remark}
 \newcommand{\mc}{\mathcal}
 \newcommand{\mf}{\mathfrak}
 \newcommand{\mcA}{\mc{A}}
 \newcommand{\mcB}{\mc{B}}
 \newcommand{\mcC}{\mc{C}}
 \newcommand{\mcD}{\mc{D}}
 \newcommand{\mcE}{\mc{E}}
 \newcommand{\mcF}{\mc{F}}
 \newcommand{\mcG}{\mc{G}}
 \newcommand{\mcH}{\mc{H}}
 \newcommand{\mcI}{\mc{I}}
 \newcommand{\mcJ}{\mc{J}}
 \newcommand{\mcK}{\mc{K}}
 \newcommand{\mcL}{\mc{L}}
 \newcommand{\mcM}{\mc{M}}
 \newcommand{\mcN}{\mc{N}}
 \newcommand{\mcO}{\mc{O}}
 \newcommand{\mcP}{\mc{P}}
 \newcommand{\mcQ}{\mc{Q}}
 \newcommand{\mcR}{\mc{R}}
 \newcommand{\mcS}{\mc{S}} 
 \newcommand{\mcT}{\mc{T}}
 \newcommand{\mcU}{\mc{U}}
 \newcommand{\mcV}{\mc{V}}
 \newcommand{\mcW}{\mc{W}}
 \newcommand{\mcX}{\mc{X}}
 \newcommand{\mcY}{\mc{Y}}
 \newcommand{\mcZ}{\mc{Z}}
 \newcommand{\mfp}{\mf{p}}
 \newcommand{\mfP}{\mf{P}}
 \newcommand{\mfq}{\mf{q}}
 \newcommand{\mfQ}{\mf{Q}} 
 \newcommand{\mfd}{\mf{d}}
 \newcommand{\mfa}{\mf{a}}
 \newcommand{\mfb}{\mf{b}}
 \newcommand{\oQ}{\overline{\QQ}}
 \newcommand{\AAA}{\mathbb{A}}
 \newcommand{\CC}{\mathbb{C}}
 \newcommand{\FF}{\mathbb{F}}
  \newcommand{\GG}{\mathbb{G}}
 \newcommand{\NN}{\mathbb{N}}
 \newcommand{\PP}{\mathbb{P}}
 \newcommand{\QQ}{\mathbb{Q}}
 \newcommand{\QQd}{\QQ^{(d)}}
 \newcommand{\RR}{\mathbb{R}}
 \newcommand{\ZZ}{\mathbb{Z}}
 \newcommand{\oneto}[1]{\{1,\dots,#1\}}
 \newcommand{\tors}[0]{\textup{tors}}
 \newcommand{\QQbar}[0]{\overline{\QQ}}

   \def\bbbone{{\mathchoice {\rm 1\mskip-4mu l} {\rm 1\mskip-4mu l}
   {\rm 1\mskip-4.5mu l} {\rm 1\mskip-5mu l}}}

 \renewcommand{\thefootnote}{\fnsymbol{footnote}}

%% file: rbintro.tex
We work within a fixed algebraic closure $\QQbar$ of $\QQ$ throughout this paper.  We write $h$ for the usual absolute logarithmic height on algebraic numbers (defined in Section \ref{lower}).  If $K$ is a subfield of $\QQbar$, then $K$ satisfies the \emph{Bogomolov property}, (B), if there exists some $\varepsilon > 0$ such that there is no element $\alpha \in K^\times$ such that $0 < h(\alpha)  < \varepsilon$.  This definition was first stated in \cite{bombierizannier}.  Recall that $h(\alpha) = 0$ if and only if $\alpha$ is a root of unity \cite[Theorem 1.5.9]{bombierigubler}.  We introduce the following generalization of (B) to relative extensions.

\begin{definition}\label{bogdef}\normalfont
Let $\QQ \subseteq K \subseteq L \subseteq \QQbar$ be fields.  We say that $L/K$ is \emph{Bogomolov}, or that $L/K$ satisfies the \emph{relative Bogomolov property}, (RB), if there exists $\varepsilon > 0$ such that
\begin{equation}\label{rebo}
\{\alpha \in L^\times ~\big|~ 0<h(\alpha)<\varepsilon\} \subseteq K.
\end{equation}
\end{definition}
\noindent The following facts are immediate from the definition.
\begin{proposition}\label{basic}
Suppose $K \subseteq L \subseteq M$ are subfields of $\QQbar$.
\begin{enumerate}
\item[(a)] If $K$ satisfies (B), then $L/K$ is Bogomolov if and only if $L$ satisfies (B);
\item[(b)] $M/K$ is Bogomolov if and only if $M/L$ and $L/K$ are both Bogomolov; and
\item[(c)] if $L\setminus M$ contains a root of unity and $M/K$ is not Bogomolov, then $L/K$ is not Bogomolov.
\end{enumerate}
\end{proposition}
Part (c) follows because multiplying an algebraic number by a root of unity does not affect the height.  In light of this observation, it would have been equivalent to define (RB) by replacing (\ref{rebo}) with
\begin{equation}\label{rebo2}
\{\alpha \in L^\times ~\big|~ h(\alpha)<\varepsilon\} \subseteq K.
\end{equation}
Note that the converse of (d) does not hold, as demonstrated in \cite[Example 5.3]{adz}, where it is shown that $\QQ^{tr}(i)/\QQ^{tr}$ is not Bogomolov.  Here $\QQ^{tr}$ denotes the maximal totally real extension of $\QQ$, which satisfies (B) \cite{schinzel}.  Interestingly, Pottmeyer \cite{pottmeyernote} has recently stated a bound that implies that every finite extension of $\QQ^{tr}(i)$ (the so-called ``maximal CM field'') satisfies (RB), using an archimedean estimate of Garza \cite{sgbound}.  In proving the following theorem we will construct examples where (RB) is and is not satisfied, for both finite and infinite extensions of a field failing (B).  These constructions will be exhibited in Section \ref{examples}.  
 
\begin{theorem}\label{4cases}
Let $K/\QQ$ be an infinite extension that does not satisfy (B), and suppose $L/K$ is an algebraic extension.  All of the following possibilities can occur:
\begin{enumerate}
\item[(a)] $L/K$ is finite and Bogomolov,
\item[(b)] $L/K$ is finite and not Bogomolov,
\item[(c)] $L/K$ is infinite and Bogomolov, and
\item[(d)] $L/K$ is infinite and not Bogomolov.
\end{enumerate}
\end{theorem}
\noindent Note that it suffices to provide examples of (b) and (c).  Pottmeyer's example above also establishes (a), but our examples are quite different, using non-archimedean information instead of archimedean.

It is natural to ask what conditions can be placed on a field $K$ of algebraic numbers to ensure that there exists at least one nontrivial relative Bogomolov extension $L/K$.  To this end we prove the following, our main result.  
\begin{theorem}\label{finram}
Let $K/\QQ$ be an algebraic extension.  Assume there exists a (finite) rational prime $\ell$ and a number field $F \subseteq K$ such that no prime of $\mcO_F$ lying over $\ell$ is ramified in $K/F$ -- in particular this holds if $K/\QQ$ is Galois and some prime $\ell$ has finite ramification index in $K$.  Then there exist nontrivial relative Bogomolov extensions $L/K$.  These extensions can be constructed explicitly of the form $K(\sqrt[\ell]{\alpha})$ for appropriately chosen elements $\alpha \in K$.
\end{theorem}
\noindent This theorem should be compared with \cite[Theorem 2]{bombierizannier}, which states that a Galois extension with bounded \emph{local degrees} (ramification index times inertial degree) has the Bogomolov property.

We briefly describe what is known on fields with the Bogomolov property in order to put our results in context.  Schinzel \cite{schinzel} showed in 1973 that there is a positive lower bound on the height of totally real numbers, establishing (B) for the maximal totally real field $\QQ^{tr}$.  This can be described as an ``archimedean'' height estimate, and was generalized by Garza to a lower bound on the height of algebraic numbers with at least one real conjugate \cite{sgbound}.  Another common approach that has been used (for example in \cite{habegger}) for archimedean estimates is equidistribution, starting with Bilu's Theorem \cite{bilu}, but these techniques will not be used in the current paper in favor of the Schinzel-Garza inequality.

One non-archimedean strategy originates in Amoroso and Dvornicich's paper \cite{amorosodvornicich2000}, where it is shown that (B) is enjoyed by the maximal abelian extension $\QQ^{ab}$ of $\QQ$, which was generalized to relative extensions and strengthened considerably in \cite{az1} and \cite{az2}.  Their strategy involves estimating how close a certain automorphism in a Galois group is to the action of raising an element to a power, with respect to a place lying over some auxilliary prime.  This strategy is quite powerful and is also used in \cite{habegger}, the elliptic curve analogue of \cite{amorosodvornicich2000}, and in \cite{adz}, where it is summarized nicely by their Lemma 2.2.  The main theorem (Theorem 1.2) of the latter paper generalizes both the results on abelian extensions and \cite[Theorem 2]{bombierizannier}, which states that (B) is satisfied by a field having bounded local degrees above some rational prime.  

In our present efforts to prove that a relative extension $L/K$ is Bogomolov when $K$ may not satisfy (B), it is not clear that the Amoroso-Dvornicich technique can be used to produce any new results.  Instead we appeal to more classical bounds in terms of ramification.  Our main tool is the lower bound \cite[Theorem 2]{silvermanlb}, due to Silverman.  This bound is written in notation more similar to ours in \cite[Section 3]{widmerN}, where the author uses it effectively to give examples of fields satisfying the closely related Northcott property, (N).  This stronger property, first defined along with (B) in \cite{bombierizannier},  is satisfied by a field $K$ if for any $T$ at most finitely points in $K$ have height at most $T$.  Silverman's inequality generalizes to the relative case a type of bound going back to the theorem \cite[Theorem 1]{mahler} of Mahler, which is exactly the lower bound used in \cite[Theorem 2]{bombierizannier}, where as mentioned before the authors exploit the existence of a bound on local degrees above some finite rational prime.  Our Theorem \ref{finram} has the related hypothesis of finite ramification above a prime -- for this theorem we also require an archimedean estimate coming from the above stated theorem of Garza.

The rest of this paper is organized as follows.  
In Section \ref{lower} we introduce notation and prove a general criterion, Theorem \ref{relbocrit}, for when we can use ramification information to conclude that a finite relative extension $L/K$ is Bogomolov.  In Section \ref{lth} we describe how to apply these techniques to bound below the heights of elements properly contained in an extension of the form $K(\sqrt[\ell]{\alpha})$, using Hecke's classical theory of ramification in Kummer extensions.  We combine this with the archimedean Schinzel-Garza inequality to prove Theorem \ref{finram}.  Finally, in Section \ref{examples} we construct the aforementioned explicit examples.  We use our ramification criterion to construct explicitly a field $K$ (which is quite different from $\QQ^{tr}(i)$) such that for each $\alpha \in K^\times$ there is a Bogomolov extension of the form $K(\sqrt[\ell]{\alpha})$.  

We conclude the introduction by mentioning a few questions for further investigation.  As mentioned above, if $L\setminus K$ contains a root of unity $\zeta$ and $K$ does not satisfy (B), then $L\setminus K$ contains elements of arbitrarly small positive height of the form $\zeta \alpha$, with $\alpha \in K$.  If one could construct such an extension where the \emph{only} elements of small height in $L\setminus K$ were obtained by multiplying elements of $K$ by roots of unity, this would suggest a weaker version of (RB) that could be explored.  Pottmeyer has shown that all finite extensions of the maximal CM field are Bogomolov.  In this same spirit, it would be interesting to exhibit fields $K \subsetneq \QQbar$ admitting \emph{no} Bogomolov extensions.  One easy example of this can be found if $K$ is the subfield of $\QQbar$ fixed by complex conjugation, i.e. $\QQbar \cap \RR$ (if we first embed $\QQbar \hookrightarrow \CC$), but one might expect this to happen for other fields $K$ that are sufficiently ``big,'' for example pseudo-algebraically closed.  If this occurs for a field $K$ satisfying (B), this field would be maximal with respect to the Bogomolov property.
\section*{Acknowledgments}
The author wishes to thank Jeffrey Vaaler and Felipe Voloch for some useful discussions related to Corollary \ref{v}, and Lukas Pottmeyer for pointing out relevant properties of $\QQ^{tr}(i)$.

%% file: rblower.tex
First we establish some notation conventions.  For a finite extension of number fields $M/F$, we write $D_{M/F}$ for the relative discriminant ideal, and $N_{M/F}$ for the relative ideal norm.  For a tower of number fields $M'/M/F$ will often use the well-known identity
\begin{equation}\label{chain}
D_{M'/F} = D_{M/F}^{[M':M]}\cdot N_{M/F}\left(D_{M'/M}\right) \text{\cite[Proposition 4.15]{narkiewicz}}.
\end{equation}
A prime $\mfp$ of $F$ will mean a prime ideal in the ring of integers $\mcO_F$, with corresponding non-archimedean valuation $v_\mfp$.  If $\pi$ is a uniformizing parameter for the associated place $v$, and if $\mfp$ divides the rational prime $\ell$, we normalize the absolute value $|\cdot|_v$ so that $|\pi|_v^{[F:\QQ]} = \ell^{f}$, where $f$ is the associated residue class degree.

The absolute logarithmic height of an algebraic number $\alpha$ is given by 
\begin{equation}\label{h}
h(\alpha) = \sum_v \log^+|\alpha|_v,
\end{equation}
the sum being taken over the places of any number field containing $\alpha$.  We denote the multiplicative height $H(\alpha) = \exp h(\alpha)$.  We will often use basic facts such as 
\cite[Lemma 1.5.18]{bombierigubler} 
and 
\cite[Proposition 1.5.15]{bombierigubler} 
without specific reference.

Let $F$ be a number field of degree $d$ over $\QQ$, and let $K/F$ be an algebraic extension.  We define
\begin{equation}\label{intrho}
\rho(K/F) = \limsup~\{\delta(M)/[M:F] ~\big|~ F \subseteq M \subseteq K, ~[M:F] < \infty\},
\end{equation}
where $\delta(M)$ denotes the number of archimedean places of $M$.  In this context the limit superior is taken over the directed set of finite subextensions of $K/F$.  In other words, $\rho(K/F)$ is the least real number $\rho$ such that for any finite extension $M/F$ contained in $K$, there is a finite extension $M'/M$ with $M' \subseteq K$ such that $\delta(M')/[M':F] \leq \rho$.  Note that $d/2 \leq \rho(K/F) \leq d$, and that $\rho(L/F) \leq \rho(K/F)$ for any tower $L/K/F$.  Of course if $K/F$ is finite, then $\rho(K/F) = \delta(K)/[K:F]$.  If $C/F$ is a finite extension of degree $s$ with relative discriminant ideal $D_{C/F}$, we define the \emph{norm excess discriminant} of $C/F$ with respect to $K$ to be
\begin{equation}\label{excess}
\mcE(K,C/F) = \inf_M ~ N_{F/\QQ} 
\left(
\frac
{D_{C/F}^e}
{\gcd \left(D_{C/F}^e,D_{M/F}^s\right)}
\right)^{1/e},
\end{equation}
Where the infimum is taken over finite extensions $M/F$ with $M \subseteq K$, and $e = [M:F]$.  The norm in (\ref{excess}) is less than or equal to  the norm of the relative discriminant of $C/F$, with equality if and only if none of the (finite) primes ramifying in $C/F$ are ramified in $M/F$.  Thus, if none of the primes ramifying in $C/F$ are ramified in $K/F$, we have
\begin{equation}
\mcE(K,C/F) = N_{F/\QQ}\left(D_{C/F}\right).
\end{equation}

Note that it is only possible to have $\mcE(K,C/F)>1$ if some prime dividing $D_{C/F}$ has bounded ramification index in $K/F$.  If $S$ is the set of primes dividing $D_{C/F}$ which have bounded ramification index in $K$, then the infimum in (\ref{excess}) is taken on when $M$ is large enough that no prime lying over $S$ ramifies in $K/M$.  

We will apply the following inequality of Silverman \cite[Theorem 2]{silvermanlb}, cf. \cite[Section 3]{widmerN} to produce a ramification criterion for (RB).

\begin{theorem}[Silverman]
If $\gamma$ generates a relative extension of number fields $B/M$, where $[B:M] = s$ and $[B:\QQ] = d$, then 
\begin{equation}\label{silvermanbound}
H(\gamma) \geq s^{-\frac{\delta(M)}{2d(s-1)}} \cdot N_{M/\QQ}\left( D_{B/M} \right)^{\frac{1}{2ds(s-1)}}.
\end{equation}
\end{theorem}

This is a relative field discriminant version of a bound of Mahler \cite[Theorem 10]{mahler}.  Widmer exploited the dependence only on relative ramification in this bound to produce a ramification criterion for the Northcott property \cite{widmerN}.  Our most general criterion for the relative Bogomolov property using this bound is the following theorem.

\begin{theorem}\label{relbocrit}
Let $K/\QQ$ be an algebraic extension, and let $\alpha$ be an algebraic number with minimal polynomial $f(x)$ over $K$.  Let $L = K(\alpha)$, and let $F$ be a number field such that $F \subseteq K$ and $[F(\alpha):F] = [L:K]$.  Let $d = [F:\QQ]$, and $F' = F(\alpha)$.  Assume that $F'$ and $K$ are linearly disjoint over $F$.

Suppose that for each field $C$ with $F \subsetneq C \subseteq F'$ we have
\begin{equation}\label{crit}
\mcE(K,C/F) > s^{\rho s},
\end{equation}
where $s = [C:F]$ and $\rho = \rho(K/F)$.  Then $L/K$ is Bogomolov.

Explicitly, if $\gamma \in L^\times \setminus K^\times$, then

\begin{equation}\label{relbocritbound}
 H(\gamma) \geq \min_{F\subsetneq C\subseteq F'}~ \left\{\left(\mcE(K,C/F)\cdot s^{-\rho s} \right)^{\frac{1}{2ds(s-1)}}\right\}.
\end{equation}
\end{theorem}
\noindent Notice that the bound in (\ref{relbocritbound}) is ``uniform'' in that it only depends on the relative discriminants and degrees of the fields involved, not on the fields themselves.  


We fix some notation for our proof of Theorem \ref{relbocrit} that will be kept throughout this section.  Let $M/F/\QQ$ be a tower of finite extensions.  Let $d = [F:\QQ]$, 
$e = [M:F]$.  
Assume $\alpha$ generates an extension $F'/F$ of degree $m > 0$ such that $F'$ and $M$ are linearly disjoint over $F$.  Let $L = M(\alpha)$.  Suppose $\gamma \in L^\times \setminus M^\times$, and let $B = M(\gamma)$ and $s = [B:M]$.  There is a canonical lattice isomorphism between the intermediate fields between $M$ and $L$ and those between $F$ and $F'$, and we let $C$ be the field corresponding to $B$.  More concretely, if $W$ denotes the Galois closure of $F'/F$, then $W$ is also linearly disjoint with $M$ over $F$.  Then there is a canonical isomorphism  (restriction of domain) 
\begin{equation}
\phi: \Gal(MW/M) \overset{\sim}{\longrightarrow} \Gal(W/F),
\end{equation}
 since $W\cap M = F$.  Galois theory tells us that corresponding to $\phi\left(\Gal(MW/B)\right)$ is an intermediate field $C$ with $F \subsetneq C \subseteq F'$ and $[C:F] = s$.  The following field diagram illustrates the situation.
\begin{equation}\label{prefalldiag}
\begin{xy}
(40,25)*+{L}="L";%
(38,10)*+{B}="B";%
(20,5)*+{M}="M"; (60,5)*+{F'}="F'";%
(52.5,-4.5)*{~}="gh1";
(50.5,-2.5)*{~}="gh2";
(58,-10)*+{C}="C";%
(40,-15)*+{F}="F";%
(40,-27)*+{\QQ}="Q";%
{\ar@{-}_s "M";"B"};{\ar@{-} "B";"L"};
{\ar@{-}_s "F";"C"};{\ar@{-} "C";"F'"};
{\ar@{-}^{m} "M";"L"};{\ar@{-}_{e} "F'";"L"};
{\ar@{-}_{e} "F";"M"};{\ar@{-}^{m} "F";"F'"};
{\ar@{-}_{d} "Q";"F"};
{\ar@{-} "C";"gh1"};
{\ar@{-}_{e} "gh2";"B"};
\end{xy}
\end{equation}

The following lemma illustrates our use of Silverman's Inequality
\begin{lemma}\label{prefall}
Using the notation and assumptions of the preceeding paragraph, we have
\begin{equation}\label{pfb}
H(\gamma) \geq s^{-\frac{\rho(M/F)}{2d(s-1)}} \cdot 
N_{F/\QQ} 
\left(
\frac{D_{C/F}^e}
{\gcd\left(D_{C/F}^e, D_{M/F}^s\right)}
\right)^{\frac{1}{2des(s-1)}}.
\end{equation}
In particular, if $D_{F'/F}$ and $D_{M/F}$ are relatively prime we have
\begin{equation}\label{rppfb}
H(\gamma) \geq s^{-\frac{\rho(M/F)}{2d(s-1)}} \cdot 
N_{F/\QQ}
\left(D_{C/F}\right)^{\frac{1}{2ds(s-1)}}.
\end{equation}
\end{lemma}
\begin{proof}
We apply Silverman's Inequality to the tower $B/M/\QQ$, obtaining
\begin{equation}\label{pfsb2}
H(\gamma) \geq s^{-\frac{\delta(M)}{2de(s-1)}} \cdot N_{M/\QQ}\left( D_{B/M} \right)^{\frac{1}{2des(s-1)}} = s^{-\frac{\rho(M/F)}{2d(s-1)}} \cdot N_{M/\QQ}\left( D_{B/M} \right)^{\frac{1}{2des(s-1)}} .
\end{equation}
Using basic properties of relative norms and discriminants, we have
\begin{align}
N_{M/\QQ}\left( D_{B/M} \right) 
= N_{F/\QQ}\left(N_{M/F}\left( D_{B/M} \right)\right)
= N_{F/\QQ}\left( \frac{D_{B/F}}{D_{M/F}^s} \right).
\end{align}
Since $D_{B/F}$ is divisible by both $D_{C/F}^{e}$ and $D_{M/F}^s$, we now have
\begin{align}\label{hihi}
N_{M/\QQ}\left( D_{B/M} \right) 
\geq N_{F/\QQ}\left( \frac{\lcm \left(D_{C/F}^{e},D_{M/F}^s\right)}{D_{M/F}^s} \right)
= N_{F/\QQ}\left( \frac{D_{C/F}^{e}}{\gcd \left(D_{C/F}^{e},D_{M/F}^s\right)} \right).
\end{align}
Combining (\ref{hihi}) with (\ref{pfsb2}) completes the proof of (\ref{pfb}), and (\ref{rppfb}) follows immediately.
\end{proof}

Now we move from the case of a finite extension $M/F$ to that of a possibly infinite extension $K/F$.  The following corollary completes our proof of Theorem \ref{relbocrit}.

\begin{corollary}\label{fall}
Let $K/\QQ$ be an algebraic extension, and let $\alpha$ be an algebraic number with minimal polynomial $f(x)$ over $K$.  Let $L = K(\alpha)$, and let $F$ be a number field such that $F \subseteq K$ and $[F(\alpha):F] = [L:K]$\footnote[2]{This is satisfied, for example, if $F$ contains the coefficients of $f(x)$}.  Let $d = [F:\QQ]$, $\rho = \rho(K/F)$ as in (\ref{intrho}), and $F' = F(\alpha)$.  Assume that $F'$ and $K$ are linearly disjoint over $F$.  If $\gamma \in L^\times \setminus K^\times$, then
\begin{equation}\label{mfb}
H(\gamma) \geq 
\min \left\{
s^{-\frac{\rho}{2d(s-1)}} \cdot 
\mcE(K,C/F)^{\frac{1}{2ds(s-1)}}
~\big|~
F\subsetneq C \subset F, ~s = [C:F] \right\}.
\end{equation}
In particular, if $D_{F'/F}$ and $D_{M/F}$ are relatively prime for all $F \subseteq M \subseteq K$ with $[M:\QQ] < \infty$, then
\begin{equation}\label{rpfb}
H(\gamma) \geq 
\min \left\{
s^{-\frac{\rho}{2d(s-1)}} \cdot 
N_{F/\QQ}
\left(D_{C/F}\right)^{\frac{1}{2ds(s-1)}}
~\big|~
F\subsetneq C \subset F, ~s = [C:F] \right\}.
\end{equation}
\end{corollary}
\begin{proof}
Let $m = [L:K]$.  Let $M/F$ be a finite extension of degree $e$ such that $M \subseteq K$ and $[M(\gamma):M] = [K(\gamma):K]$.  We are now in the setting to apply Lemma \ref{prefall}, so we have inequalities (\ref{pfb}) and (\ref{rppfb}), where $C$ and $s$ are as defined in the paragraph preceeding Lemma \ref{prefall}.  Notice that $\rho(M/F) \leq \rho$ and
\begin{equation}
 N_{F/\QQ} 
\left(
\frac{D_{C/F}^e}
{\gcd\left(D_{C/F}^e, D_{M/F}^s\right)}
\right)^{1/e} \leq  
\mcE(K,C/F),
\end{equation}
by definition.  The corollary follows, as does Theorem \ref{relbocrit}.
\end{proof}


%% file: rblth.tex
Extensions formed by adjoining an $\ell \nth$ root of an element, where $\ell$ is a prime, are an easy source of examples in which we can successfully apply the bounds of the previous section.  An extension of prime degree have no intermediate extensions, so application of Theorem \ref{relbocrit} becomes much cleaner.  Furthermore, the discriminants of such extensions when the base field contains a primitive $\ell \nth$ root of unity (Kummer extensions) are completely understood thanks to classical work of Hecke (see \cite[\S 39]{Hecke}, cf. \cite[Section 10.2.3]{cohen2}).  We now illustrate how we can exploit this theory, even when the base field does not contain an $\ell \nth$ root of unity.

For the next lemma and its corollaries we use the following setup.  Let $F/\QQ$ be a finite
extension of degree $d$, let $\alpha \in \mcO_F$, and let $\ell$ be a rational prime.  Assume $\alpha$ is not an $\ell \nth$ power in $F$, which means that $f(x) = x^\ell - \alpha$ is irreducible over $F$.  To see this, note that a factorization of $f(x)$  over $F$ implies that $F$ contains elements of the form $\alpha^{m/\ell}$ and $\alpha^{n/\ell}$, such that a quotient of appropriate powers of these elements is a root of $f(x)$.  Let $F' = F(\alpha^{1/\ell})$ for some choice of the root.  Assume that $\ell \mcO_F$ and $\alpha \mcO_F$ are relatively prime ideals.  In the following lemmas $\mfp$ will always denote a prime of $F$ lying over $\ell$ with residue class degree $f(\mfp|\ell) = [\mcO_F/\mfp : \ZZ/\ell \ZZ]$ and ramification index $e(\mfp|\ell)$. For each $\mfp | \ell$ define $a(\mfp)$ to be the greatest integer $k$ such that the congruence   
\begin{equation}\label{acongruence}
x^\ell-\alpha \equiv 0 \imod {\mfp^k}
\end{equation}
has a solution in $F$.  

\begin{lemma}\label{acolem}
Let $\rho$ be a real number with $\frac{1}{2} \leq \rho < 1$.  If for each $\mfp | \ell$ we have 
\begin{equation}\label{acond}
a(\mfp) \leq 1 + \ell(1-\rho),
\end{equation}
then we have 
\begin{equation}\label{lrhol}
\ell^{\lceil \rho \ell \rceil} \mcO_F ~\big|~ 
D_{F'/F},
\end{equation}
where $\lceil x \rceil$ denotes the least integer greater than or equal to the real number $x$.
\end{lemma}

To simplify application of this lemma, we will prove the following two corollaries.

\begin{corollary}\label{v}
Let $\rho$ be a real number with $\frac{1}{2} \leq \rho < 1$.  Assume that for each $\mfp$ we have that $f(\mfp|\ell) = 1$.  There exists a constant $c$ depending only on $d$ and $\alpha$ such that if $\ell \geq c$, then 
\begin{equation}\label{lrholf1}
\ell^{\lceil \rho \ell \rceil} \mcO_F ~\big|~ 
D_{F'/F},
\end{equation}
as in Lemma \ref{acolem}.
\end{corollary}

\begin{corollary}\label{a=1cor}
Suppose that for each $\mfp | \ell$ we have
\begin{equation}\label{a=1}
v_\mfp(\alpha^{\ell^{f}-1}-1)= 1,
\end{equation}
where $f = f(\mfp|\ell)$.  Then each of the primes $\mfp$ is totally ramified in $F'/F$, and
\begin{equation}\label{ll}
\ell^{\ell} \mcO_F ~\big|~ 
D_{F'/F}.
\end{equation}
\end{corollary}

\begin{proof}[Proof of Lemma \ref{acolem}]
Let $\mfp$ be a prime of $F$ lying over $\ell$ with ramification index $e = e(\mfp|\ell)$ and residual degree $f = f(\mfp|\ell)$.  First, assume that $F$ contains a primitive $\ell \nth$ root of unity, and notice that this means that
\begin{equation}\label{el}
\ell-1 | e.
\end{equation}
By \cite[Theorem 10.2.9 (3)]{cohen2}, (\ref{acond}) implies that $\mfp$ is totally ramified in $F'/F$, and we have
\begin{equation}\label{heckesay}
v_\mfp\left(D_{F'/F}\right) = (\ell-1)(\ell \frac{e(\mfp|\ell)}{\ell-1} +1-a(\mfp)).
\end{equation}
Combining (\ref{acond}) with (\ref{el}), we certainly have
\begin{equation}\label{acondb}
a(\mfp) \leq 1 + \frac{\ell e (1 - \rho)}{\ell - 1}, 
\end{equation}
and combining this with (\ref{heckesay}) yields
\begin{equation}\label{heckesaynow}
v_\mfp\left(D_{F'/F}\right) \geq e\cdot \rho \ell.
\end{equation}
This means that
\begin{equation}
D_{F'/F} \subseteq \prod_\mfp \mfp^{e(\mfp|\ell) \cdot \lceil \rho \ell \rceil} = \ell^{\lceil \rho \ell \rceil} \mcO_F,
\end{equation}
and we now have (\ref{ll}) in the case where $F$ contains a primitive $\ell \nth$ root of unity.  In general, let $F'' = F(\zeta_\ell)$, where $\zeta_\ell$ is a primitive $\ell \nth$ root of unity.  Let $n = [F'':F]$.  We have $n | \ell-1$, and so $F'' \cap F' = F$ and $[F''F':F'] = n$.  We claim that
\begin{equation}\label{discfall}
N_{F''/F}\left(D_{F''F'/F''}\right) ~\big|~ D_{F'/F}^n.
\end{equation}
An easy way to see this is by considering relative different ideals as generated by the differents of elements, as in \cite[Theorem 4.16]{narkiewicz}.  Since $F'/F$ and $F''F'/F''$ are generated by the same polynomial, it is clear that 
\begin{equation}\label{diffies}
\mcD_{F'/F} \mcO_{F''F'} \subseteq \mcD_{F''F'/F''}.
\end{equation}
Taking the norm $N_{F''F'/F}$ of both sides of (\ref{diffies}) yields (\ref{discfall}).  Our previous argument shows that $D_{F''F'/F''}$ is divisible by $\ell^{\lceil \rho \ell \rceil} \mcO_{F''}$, and so (\ref{discfall}) gives us that $D_{F'/F}^n$ is divisible by $\ell^{\lceil \rho \ell \rceil \cdot n} \mcO_F$, and take $n\nth$ roots.
\end{proof}
\begin{proof}[Proof of Corollary \ref{a=1cor}]
Since $(\mcO_F/\mfp)^\times$ has order $\ell^f-1$, we have $\alpha^{\ell^f-1} - 1 \equiv 0 \imod \mfp$, so $x_1 = \alpha^{\ell^{f-1}}$ is a solution to (\ref{acongruence}) for $k=1$.  We have 
\begin{equation}
v_\mfp(x_1^\ell - \alpha) = v_\mfp\left(\alpha^{\ell^f}-\alpha \right) = v_\mfp\left(\alpha^{\ell^f-1}-1\right).                                                                                                                                                                                            \end{equation}
Now, by \cite[Proposition 10.2.13]{cohen2}, there is no solution to (\ref{acongruence}) if $v_\mfp\left(\alpha^{\ell^f-1}-1\right) < k < \ell + 1$.  If we assume (\ref{a=1}) holds, this means that $a(\mfp) = 1$ for all $\mfp|\ell$, and the corollary follows directly from Lemma \ref{acolem}.
\end{proof}

\begin{proof}[Proof of Corollary \ref{v}]
We now assume $f(\mfp|\ell) = 1$ for all $\mfp | \ell$.  The corollary will be proved once we show that, for $\ell$ sufficiently large, the inequality (\ref{acond}) is satisfied.  It is well known that, for an algebraic number $\beta$ and a finite set of places $S$ of a field containing $\beta$, we have
\begin{equation}\label{liouville}
-h(\beta) \leq \sum_{v\in S} \log |\beta|_v \leq h(\beta) \text{\cite[(1.8, p. 20)]{bombierigubler}}.
\end{equation}
We fix a prime $\mfp |\ell$ and take the set $S$ to consist of only that place $v$ corresponding to $\mfp$.  We have $h(\alpha^{\ell-1}-1) \leq \log 2 + (\ell-1) h(\alpha)$, and so the left-hand inequality of (\ref{liouville}) yields
\begin{equation}
-\log 2 - (\ell-1) h(\alpha) \leq |\alpha^{\ell-1}|_v = -\frac{\log \ell}{d}\cdot v_\mfp(\alpha^{\ell-1}-1),
\end{equation}
so that
\begin{equation}
v_\mfp(\alpha^{\ell-1}-1) \leq \frac{d \log 2}{\log \ell} + \frac{(\ell-1)d\cdot h(\alpha)}{\log \ell} \leq 1 + \ell(1-\rho) 
\end{equation}
if $\ell$ is sufficiently large in terms of $d$ and $\alpha$.  As in the previous proof, the congruence (\ref{acongruence}) has no solution in $F$ if $v_\mfp\left(\alpha^{\ell-1}-1\right) < k < \ell + 1$, and so we have $a(\mfp) = v_\mfp(\alpha^{\ell-1}-1)$, establishing the inequality (\ref{acond}) and completing the proof.
\end{proof}

\begin{proof}[Proof of Theorem \ref{finram}]
Let $F$ be a number field of degree $d$ over $\QQ$, and let $\ell$ be a rational prime.  Let $\mfp_1,\dots \mfp_n$ be the primes of $F$ lying above $\ell$.
For $i = 1 \dots n$, let $\beta_i$ be a nontrivial element of $(\mcO_F/\mfp_i^2)^\times$ such 
that $\beta^{\ell^f} \equiv 1$.  By the Chinese Remainder Theorem we can find an element $\alpha \in \mcO_F$ such $\alpha \equiv \beta_i \imod{\mfp_i}$ for each $i$.  Therefore we have 
\begin{align}
v_{\mfp_i}(\alpha^{\ell^f-1}-1) &= 1, ~\text{for}~i = 1\dots n.\label{acheck}
\end{align}
Let $F' = F(\sqrt[\ell]{\alpha})$ for some choice of the root.  Using Corollary \ref{a=1cor} we have that $f(x) = x^\ell - \alpha$ is irreducible over $F$, each of the primes $\mfp_i$ is totally ramified in $F'/F$, and
\begin{equation}\label{llq}
\ell^{\ell} ~\big|~ D_{F'/F},
\end{equation}
and therefore 
\begin{equation}
N_{F/\QQ}\left(D_{F'/F}\right) \geq \ell^{\ell d}.
\end{equation}

We are now ready to prove theorem \ref{finram}.  Let $K/\QQ$ be an algebraic extension such that $\ell$ has bounded ramification indices in $K$, and set $F$ to be a number field of degree $d$ over $\QQ$ such that $F\subseteq K$ and no primes of $F$ lying over $\ell$ are ramified in $K/F$.  Let $\alpha \in \mcO_F$ satisfy (\ref{acheck}) as above, and let $L = K(\sqrt[\ell]{\alpha})$.  We want to show that $L/K$ is Bogomolov.  Let $F'$ be as above.  If $\mfp$ is a prime of $F$ lying over $\ell$, we know $\mfp$ is unramified in $K/F$ and totally ramified in $F'/F$, so we have that $K$ and $F'$ are linearly disjoint over $F$.

First suppose that $\rho(K/\QQ) < 1$, so that $\rho := \rho(K/F) < d$.  Our construction ensures now that
\begin{equation}\label{almosthere}
\mcE(K,F'/F) \geq \ell^{d\ell} > \ell^{{\rho}\ell}, 
\end{equation}
and therefore $L/K$ is Bogomolov by Theorem \ref{relbocrit}.  More specifically, let $\gamma \in L \setminus K$, so that $\ell = [K(\gamma):K]$.  Let $M/F$ be a finite extension such that $M \subseteq K$ and $[M(\gamma):M] = \ell$.  Then by following the proof of Theorem \ref{relbocrit}, we have 
\begin{equation}\label{nonbound}
H(\gamma) \geq \left\{ \ell^{d \ell} \cdot \ell^{-\rho(M/F) \cdot \ell} \right\}^{\frac{1}{2d\ell(\ell-1)}} \geq \ell^{\frac{d-\rho(K/F)}{2d(\ell-1)}} > 1,
\end{equation}
and in this case we are done using only our ramification criterion.

If $\rho = d$, we will have to use the following archimedean estimate of Garza.
\begin{theorem}[Garza \cite{sgbound}]
Let $K$ be a number field of degree $d$ over $\QQ$ with $r$ real places and $r'$ complex places.  If $K = \QQ(\gamma)$, then  
\begin{equation}\label{sgform}
H(\gamma) \geq \left( 2^{-d/r} + \sqrt{1+4^{-d/r}} \right)^{\frac{r}{2d}}.
\end{equation}
\end{theorem}

Now we fix an arbitrary real number $\theta \in \left(\frac{2\ell-1}{2\ell}, 1\right)$.  If $\rho(M/\QQ) \leq \theta$, then as in (\ref{nonbound}) we have
\begin{equation}\label{nonbound2}
 H(\gamma) > \ell^{\frac{(1-\theta)}{2(\ell-1)}} > 1.
\end{equation}

On the other hand, if $\rho(M/\QQ) > \theta$, let $r$ and $s$ denote the number of real and complex archimedean places of $M$, respectively.  Notice that $M(\gamma) = M(\sqrt[\ell]{\alpha})$ has $r$ real places and $\frac{r(\ell-1)}{2} + s\ell$ complex places.  This means that 
\begin{align}
\rho (\QQ(\gamma)/\QQ) &\geq \rho(M(\gamma)/\QQ) = \frac{r +\frac{r(\ell-1)}{2} + s\ell}{\ell[M:\QQ]}\\
 &= \rho(M/\QQ) - \frac{r_M}{[M:\QQ]} \cdot \left( \frac{\ell-1}{2\ell} \right) > \theta - \frac{\ell-1}{2\ell}.
\end{align}
If $\QQ(\gamma)$ has $r'$ real places and $s'$ complex places, then we now have that
\begin{equation}
\frac{r'}{d'} = 2 \rho(\QQ(\gamma)/\QQ) - 1 > 2\theta - 1 - \frac{\ell-1}{\ell} > 0
\end{equation}
by our choice of $\theta > \frac{2\ell-1}{2\ell}$.

Now we may bound below the height of $\gamma$ by an absolute constant using (\ref{sgform}).  Explicitly, writing $\phi = 2\theta-1 - \frac{\ell-1}{\ell}$, we have
\begin{equation}\label{archbound}
H(\gamma) \geq \left( 2^{-\frac{1}{\phi}} + \sqrt{1+4^{-\frac{1}{\phi}}} \right)^{\frac{\phi}{2}} > 1.
\end{equation}

Either (\ref{nonbound2}) or (\ref{archbound}) must hold.  Taking the minimum of these bounds, we see that $H(\gamma)$ is bounded below by a constant greater than 1 which depends only on $\ell$, and our proof is complete.

\end{proof}

%% file: rbexamples.tex
We begin with two examples that give us Theorem \ref{4cases}.

\begin{example}[$L/K$ not Bogomolov]\label{norb}
\normalfont
Let $b$ be a nonsquare rational number and let $K = \QQ(b^{1/2},b^{1/4},b^{1/8},\dots)$, for any choices of the roots.  Notice that $b^{1/3} \not \in K$, and let $L = K(b^{1/3})$.  The extension $L/K$ is not Bogomolov.   To see this, consider the elements $b^{1/3}b^x \in L\setminus K$, where $x$ is a rational number close to $-1/3$ with denominator a power of $2$.  Notice that $h(b^{1/3}b^x) = h(b^{x+\frac{1}{3}} = (x + \frac{1}{3})h(b) \to 0$ as $x \to -1/3$.  Many similar examples can be constructed easily, including of course infintie relative extensions.
\end{example}  

\begin{example}[$L/K$ Bogomolov]\label{yesrb}
\normalfont
Let $K = \QQ(3^{1/3},3^{1/9},3^{1/27},\dots)$, and note that $3$ is the only rational prime that ramifies in $K$.  Let $3 < p_1 < p_2 < \cdots$ be an infinite sequence of primes congruent to $3 \imod 4$.  Set $K_0 = K$, and for each $n \geq 1$ set $K_n = K_{n-1}(\sqrt{p_n})$.  For a given $n \geq 1$, we wish to apply Lemma \ref{prefall} to estimate the height of an element $\gamma \in K_n^\times \setminus K_{n-1}$.  We set $F = \QQ$ and choose $M\subseteq K_{n-1}$ to be a number field containing $\sqrt{p_1}, \sqrt{p_2},\dots,\sqrt{p_{n-1}},$ and the coefficients for the minimal polynomial of $\gamma$ over $K_n$.  We have the trivial estimate $\rho(M/F) \leq d$, so applying (\ref{rppfb}), we have
\begin{equation}
H(\gamma) \geq 2^{-\frac{1}{2}}\cdot p_n^{\frac{1}{4}} = \left(\frac{p_n}{4}\right)^{\frac{1}{4}} \geq \left(\frac{p_1}{4}\right)^{\frac{1}{4}}. 
\end{equation}
Letting $L = \cup_n K_n$, we now have that $L/K$ is an infinite Bogomolov extension, and in fact $L$ can be constructed so that the lower bound on the height of an element of $L^\times \setminus K^\times$ is arbitrarily large.  This and Example \ref{norb} establish Theorem \ref{4cases}, as in both cases the base field $K$ clearly does not satisfy (B).

If the roots $3^{1/3^i}$ are chosen in a compatible way (e.g. if we fix an embedding $\QQbar \hookrightarrow \CC$ and impose that the roots are all real), then $K$ has the property that all of its proper subfields are finite extensions of $\QQ$.  (The interested reader will verify that the only subfields of $\QQ(3^{1/3^n})$ are $\QQ(3^{1/3^i})$, $0 \leq i \leq n$.)  Therefore $K$ is not a Bogomolov extension of any subfield.
\end{example}


\begin{example}
\normalfont
We now construct a Galois extension $K/\QQ$ which does not satisfy (B), but with the property that for every element $\alpha \in K^\times$, there is an integer $\ell$ such that $K(\sqrt[\ell]{\alpha})/K$ is Bogomolov.  As mentioned before, the maximal CM field $\QQ^{tr}(i)$ also enjoys this property -- in fact, all finite extensions of this field are Bogomolov.  Our example is of a different sort, as it is generated by polynomials with no real roots, which prevents us from using archimedean estimates.  
It will suffice to consider $\alpha$ an algebraic integer, as any extension $K(\sqrt[\ell]{\alpha})$ could have been generated in this way by clearing denominators.

Set $K_0 = \QQ$, $b_1 = 12$, and let $K_{n}$ be generated over $K_{n-1}$ by adjoining all of the roots of
\begin{equation}
f_n(x) = x^{b_n} + x + 1,
\end{equation}
 where $b_n$ is a multiple of 12 chosen so that the discriminant of $f_n$ is not divisible by any of the rational primes which are ramified in $K_{n-1}$.  This is easily accomplished, since the discriminant of $f_n$ is of the form $-\big(b_n^{b_n} + (b_n-1)^{(b_n-1)\big)}$ (see \cite{masserdisc} for example -- we have used that $4|b_n$), which will not be divisible by any prime dividing $b_n$.  By the Chebotarev Density Theorem, any number field has infinitely many rational primes that are totally split in it.  Therefore we can also choose a new prime $\ell_n$ at each step such that $\ell_n$ splits completely in $K_{n-1}$ (and hence in $K_m$ for all $m < n$), and also choose $b_n$ to be divisible by $\ell_1\ell_2\cdots\ell_n$.  In this way, we obtain an infinite set of primes $S = \{\ell_1,\ell_2,\dots\}$, none of which is ramified in any of the fields $K_n$.  Furthermore, for any $n$ we have that $S$ contains arbitrarily large primes that split completely in $K_n$.  Let $K = \cup_{n=1}^\infty K_n$.  By making each $b_n$ divisible by $3$, we ensure that $f_n(x)$ is irreducible over $\QQ$ with symmetric Galois group \cite[Theorem 1]{osada}, \cite[Theorem 1]{selmertri}.  By making each $b_n$ even, we ensure that each polynomial $f_n(x)$ has no real roots, so we have $\rho(K/\QQ) = \frac{1}{2}$.  
Using basic facts about the height, we see that if $\alpha$ is a root of $f_n(x)$, we have
\begin{equation}
b_n\cdot h(\alpha) = h(\alpha^{b_n}) = h(\alpha+1) \leq \log2 + h(\alpha) + h(1) = \log2 + h(\alpha),
\end{equation}
which yields
\begin{equation}\label{yey}
0 < h(\alpha) \leq \frac{\log2}{b_n-1},
\end{equation}
which shows that $K$ does not satisfy (B).

For an odd prime $\ell \in S$, let $\beta_\ell$ denote some root of $x^\ell - \alpha$.  We assume without loss of generality that $\alpha$ is not an $\ell\nth$ power in $K$, or else we could replace $\alpha$ by an $\ell\nth$ root until this hypothesis is satisfied.  By the construction of $K$, we know that $\QQ(\alpha)$ is contained in a finite Galois extension of whose Galois group $G$ is a subdirect product of symmetric groups of degree at least 6.  This implies that $x^\ell-\alpha$ is irreducible over $K$, as its Galois group $H$ (over its field of definition) is solvable of odd prime degree.  (There can be no surjection from $G$ onto a nontrivial quotient of $H$.)  For any $m$ such that $\alpha \in K_m$, the only rational primes possibly ramifying in $K_m(\beta_\ell)$ will be those lying below prime divisors of $\alpha\mcO_{K_m}$, those ramified in $K_m$, and $\ell$.  Since $\ell$ does not ramify in any field $K_n$, we can choose $m = m(\alpha)$ large enough that $\alpha \in K_m$,  and such that no 
primes ramifying in $K_m(\beta_\ell)$ divide the discriminant of $f_n$ for any $n>m$, for any $\ell \in S$.  As described above, we know that $K_m(\beta_\ell)$ and $K$ are linearly disjoint over $K_m$.  Let $d = [K_m:\QQ]$.

Now by Corollary \ref{v} we can fix an $\ell \in S$ such that $\ell$ splits completely in $K_m$ and is large enough so that 
\begin{equation}\label{divide4}
\ell^{\lceil \rho \ell \rceil}\mcO_{K_m} ~\big|~ D_{K_m(\beta)/K_m},
\end{equation}
where $\beta = \beta_\ell$, and $\rho = \frac{3}{4}$, say.  Since none of the primes ramifying in $K_m(\beta)/K_m$ are ramified in $K/K_m$, we have 
\begin{equation}\label{krazye}
\mcE(K, K_m(\beta)/K_m) > \ell^{\ell d/2} = \ell^{\ell \rho(K/F)}.
\end{equation}
Theorem \ref{relbocrit} shows that $K(\beta)/K$ is Bogomolov.

\end{example}